\documentclass[11pt]{amsart}
\usepackage[margin=1.15in]{geometry}
\usepackage{amsmath,amssymb,amsthm}
\usepackage{tikz}
\usepackage{booktabs}
\usepackage{enumitem}
\usepackage[colorlinks=true,linkcolor=blue!50!black,citecolor=blue!50!black,urlcolor=blue!50!black]{hyperref}

\newtheorem{theorem}{Theorem}[section]
\newtheorem{proposition}[theorem]{Proposition}
\newtheorem{lemma}[theorem]{Lemma}
\newtheorem{corollary}[theorem]{Corollary}
\newtheorem{question}[theorem]{Question}
\newtheorem*{thmA}{Result A}
\newtheorem*{thmB}{Result B}
\newtheorem*{thmC}{Result C}
\theoremstyle{definition}
\newtheorem{definition}[theorem]{Definition}
\newtheorem{construction}[theorem]{Construction}
\newtheorem{example}[theorem]{Example}
\theoremstyle{remark}
\newtheorem{remark}[theorem]{Remark}

\newcommand{\Z}{\mathbb{Z}}
\newcommand{\R}{\mathbb{R}}
\newcommand{\word}[1]{(\,#1\,)}
\newcommand{\Tp}{\mathsf{T}^{+}}
\newcommand{\Tm}{\mathsf{T}^{-}}
\newcommand{\T}{\mathsf{T}}
\newcommand{\X}{\mathsf{X}}
\newcommand{\Rot}{\mathsf{R}}
\newcommand{\Sym}{\mathsf{S}}
\newcommand{\M}{\mathcal{M}}

\title{Some petal diagrams of the unknot are hard}
\author{Alexei Vernitski}
\address{[Address]}
\email{[email]}
\subjclass[2020]{57K10}
\keywords{petal diagram, petal number, unknot, Reidemeister moves, arc presentation, grid diagram, monotonic simplification}

\begin{document}

\begin{abstract}
A petal diagram of a knot is a projection with a single multi-crossing and no nested loops; it is encoded by a permutation of the heights of the strands through the multi-crossing. Colton, Glover, Hughes and Sandberg proved a Reidemeister-type theorem for petal diagrams: two petal permutations represent the same knot if and only if they are related by trivial petal additions and deletions and by crossing exchanges. We ask whether every petal diagram of the unknot can be reduced to the one-petal diagram without ever increasing the number of petals, in analogy with Dynnikov's monotonic simplification theorem for rectangular diagrams. By an exhaustive, certified computer search we show that this is true for diagrams with at most $7$ petals and false for $9$ petals. Of the $40320$ petal diagrams with $9$ petals, $24992$ represent the unknot, and exactly $108$ of them are \emph{hard}: none of them admits a crossing exchange or a trivial petal deletion, even if two natural petal-number-preserving symmetries are allowed. Up to these symmetries and mirror image there are three hard diagrams. Two of them can be untangled by passing through $11$ petals; the third cannot be untangled through diagrams with at most $11$ petals, but can through $13$. We explain why the phenomenon differs from the rectangular case: a petal diagram is an arc presentation whose cyclic order of pages is determined by the order of its vertices on the binding, and no elementary move of Cromwell and Dynnikov preserves this rigid structure.
\end{abstract}

\maketitle

\section{Introduction}

A \emph{petal diagram} (or petal projection) of a knot $K$ is a projection of $K$ to a plane whose only multiple point is a single multi-crossing through which every strand passes straight, and whose complement of the multi-crossing consists of $2n+1$ loops, the \emph{petals}, none of which is nested inside another. Adams et al.\ \cite{Adams1} introduced petal diagrams and proved that every knot has one. The heights of the $2n+1$ strands through the multi-crossing, listed in the order in which the knot traverses them, form a permutation of $\{0,1,\dots,2n\}$ which determines the knot; see Figure~\ref{fig:examples}. The minimal number of petals is the \emph{petal number} $p(K)$ \cite{Adams1,Adams2}. Petal diagrams also underlie the Petaluma model of random knots \cite{EZ,EHLN1,EHLN2}, and petal grid diagrams have been used to compute petal numbers of torus knots \cite{LeeLee}.

Since many permutations represent the same knot, it is natural to ask for a set of combinatorial moves on petal permutations playing the role of the Reidemeister moves. Colton, Glover, Hughes and Sandberg \cite{CGHS} defined two such moves, \emph{trivial petal additions} (with their inverses, deletions) and \emph{crossing exchanges}, and proved that they connect any two petal permutations of the same knot. A trivial petal addition increases the number of petals by two; a crossing exchange preserves it.

This raises a monotonicity question. For classical diagrams it is well known that untangling an unknot diagram by Reidemeister moves may require temporarily increasing the number of crossings; such diagrams are called \emph{hard} \cite{Goeritz,Ochiai,KL,KL2,HK,BCLMMSSS,ABDEJLTZ}. For rectangular (grid) diagrams the situation is opposite: Dynnikov \cite{Dynnikov} proved that every rectangular diagram of the unknot can be reduced to the trivial one by exchange moves and destabilizations alone, i.e.\ without increasing the complexity. A similar monotonicity holds for closed braid representatives of the unlink, by work of Birman and Menasco \cite{BM}. Where do petal diagrams fit?

\begin{definition}\label{def:hard}
Let $\M$ be a set of moves on petal permutations, each of which preserves the knot type. A petal permutation $p$ of the unknot with $N$ petals is \emph{$\M$-hard} if there is no sequence of moves in $\M$ from $p$ to the one-petal permutation $\word{0}$ all of whose intermediate permutations have at most $N$ petals. The \emph{excess} $e_\M(p)$ is the smallest $E\ge 0$ such that there is such a sequence all of whose intermediate permutations have at most $N+E$ petals.
\end{definition}

The excess is the petal analogue of the number of ``extra crossings'' studied in \cite{BCLMMSSS}. Throughout, $\M$ consists of the moves of \cite{CGHS} together with two symmetries which are evidently isotopies and preserve the number of petals:
\[
\M=\{\T^{\pm}, \X, \Rot, \Sym\},
\]
where $\T^{\pm}$ are trivial petal additions/deletions, $\X$ are crossing exchanges, $\Rot$ is the cyclic rotation of heights and $\Sym$ is the rotation of $\R^3$ by $\pi$ about a line in the projection plane. All moves are defined precisely in Section~\ref{sec:moves}. Since $N$ is always odd and only $\T^\pm$ changes $N$ (by $\pm2$), the excess is always even. Enlarging the move set can only decrease hardness, so including $\Rot$ and $\Sym$ makes our hardness statements stronger.

Our results are as follows. We state them as \emph{Results} rather than theorems to stress that their proofs are entirely computational: each consists of an exhaustive, certified computer search, described in Section~\ref{sec:computation}, and not of a conceptual argument.

\begin{thmA}
Every petal permutation of the unknot with at most $7$ petals can be reduced to $\word{0}$ by crossing exchanges and trivial petal deletions only. In particular no such permutation is $\M$-hard.
\end{thmA}

\begin{thmB}
Among the $8!=40320$ petal permutations with $9$ petals (counted up to cyclic rotation), $24992$ represent the unknot and exactly $108$ are $\M$-hard. They form six orbits of size $18$ under the dihedral group generated by $\Rot$ and $\Sym$, and three orbits up to mirror image, with representatives
\[
\alpha=\word{0\,2\,4\,1\,5\,7\,3\,8\,6},\qquad
\beta=\word{0\,2\,5\,3\,8\,1\,7\,4\,6},\qquad
\gamma=\word{0\,2\,4\,7\,3\,8\,6\,1\,5}.
\]
Each of the $108$ permutations is \emph{rigid}: no crossing exchange and no trivial petal deletion applies to it. Consequently $9$ is the smallest number of petals of an $\M$-hard petal permutation.
\end{thmB}

\begin{thmC}
The hard permutations in the orbits of $\alpha$, $\beta$ and their mirror images have excess $2$. The hard permutations in the orbit of $\gamma$ and of its mirror image have excess $4$: the set of all permutations reachable from $\gamma$ by moves in $\M$ through permutations with at most $11$ petals is finite, has exactly $3186$ elements ($18$ with $9$ petals and $3168$ with $11$ petals), and contains none with fewer than $9$ petals.
\end{thmC}

Explicit untangling sequences realizing the excess are given in Section~\ref{sec:results}. Remarkably, they use only the moves $\T^\pm$ and $\X$ of \cite{CGHS}, without $\Rot$ or $\Sym$. If one restricts to the moves $\{\T^\pm,\X\}$ of \cite{CGHS}, Result~A still holds, and the number of hard $9$-petal permutations rises to $264$ (Remark~\ref{rem:TX}).

A robust consequence of Result~C is that hardness of $\gamma$ cannot be removed by adding ``short'' composite moves: any move obtained by composing moves of $\M$ through diagrams with at most two more petals than the starting diagram still fails to untangle $\gamma$ monotonically (Corollary~\ref{cor:robust}).

As said, the proofs of Results A--C are entirely computational. Knot types are certified in both directions: unknottedness by an explicit simplification to a crossingless diagram, and knottedness by a non-trivial Alexander polynomial. The search itself is exhaustive. The code is provided as an ancillary file (Section~\ref{sec:computation}). We emphasize that our results use only the (comparatively easy) fact that each individual move preserves the knot type; they do not depend on the completeness part of the theorem of \cite{CGHS}.

In Section~\ref{sec:pages} we explain the contrast with Dynnikov's theorem. A petal diagram is a special arc presentation (Adams et al.\ \cite[\S3]{Adams1} call these \emph{petal arc presentations}). In a general arc presentation, the cyclic order of the vertices on the binding and the cyclic order of the pages are independent, and Dynnikov's exchange moves permute either of them. In a petal arc presentation the page in which each arc lies is forced by the position of the arc along the knot. We show that none of the elementary moves of Cromwell and Dynnikov preserves this structure (Proposition~\ref{prop:noexchange}); a crossing exchange is a composite of two vertex swaps, neither of which is an exchange move by itself.

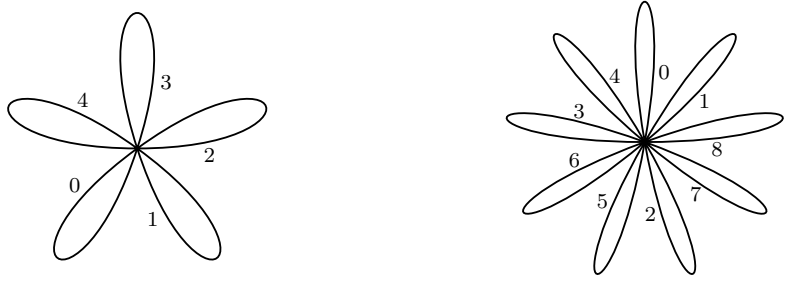
\begin{figure}[t]
\centering
\begin{tikzpicture}[scale=0.62,line width=0.7pt]
\draw (0,0) .. controls (72.00:4.05) and (108.00:4.05) .. (0,0);
\draw (0,0) .. controls (144.00:4.05) and (180.00:4.05) .. (0,0);
\draw (0,0) .. controls (216.00:4.05) and (252.00:4.05) .. (0,0);
\draw (0,0) .. controls (288.00:4.05) and (324.00:4.05) .. (0,0);
\draw (0,0) .. controls (360.00:4.05) and (396.00:4.05) .. (0,0);
\node[font=\scriptsize] at (66.50:1.55) {3};
\node[font=\scriptsize] at (282.50:1.55) {1};
\node[font=\scriptsize] at (138.50:1.55) {4};
\node[font=\scriptsize] at (354.50:1.55) {2};
\node[font=\scriptsize] at (210.50:1.55) {0};
\fill (0,0) circle (1.6pt);
\end{tikzpicture}
\qquad\qquad
\begin{tikzpicture}[scale=0.62,line width=0.7pt]
\draw (0,0) .. controls (80.00:4.05) and (100.00:4.05) .. (0,0);
\draw (0,0) .. controls (120.00:4.05) and (140.00:4.05) .. (0,0);
\draw (0,0) .. controls (160.00:4.05) and (180.00:4.05) .. (0,0);
\draw (0,0) .. controls (200.00:4.05) and (220.00:4.05) .. (0,0);
\draw (0,0) .. controls (240.00:4.05) and (260.00:4.05) .. (0,0);
\draw (0,0) .. controls (280.00:4.05) and (300.00:4.05) .. (0,0);
\draw (0,0) .. controls (320.00:4.05) and (340.00:4.05) .. (0,0);
\draw (0,0) .. controls (360.00:4.05) and (380.00:4.05) .. (0,0);
\draw (0,0) .. controls (400.00:4.05) and (420.00:4.05) .. (0,0);
\node[font=\scriptsize] at (74.50:1.55) {0};
\node[font=\scriptsize] at (274.50:1.55) {2};
\node[font=\scriptsize] at (114.50:1.55) {4};
\node[font=\scriptsize] at (314.50:1.55) {7};
\node[font=\scriptsize] at (154.50:1.55) {3};
\node[font=\scriptsize] at (354.50:1.55) {8};
\node[font=\scriptsize] at (194.50:1.55) {6};
\node[font=\scriptsize] at (394.50:1.55) {1};
\node[font=\scriptsize] at (234.50:1.55) {5};
\fill (0,0) circle (1.6pt);
\end{tikzpicture}
\caption{Left: a $5$-petal diagram of the trefoil with petal permutation $\word{3\,1\,4\,2\,0}$. Right: the hard $9$-petal diagram $\gamma=\word{0\,2\,4\,7\,3\,8\,6\,1\,5}$ of the unknot. Each straight strand through the centre carries its height, written next to the edge where the strand enters a petal on its clockwise side. Following the knot, one meets the strands in the cyclic order of the permutation and passes from each petal to the petal $n+1$ steps further counterclockwise.}
\label{fig:examples}
\end{figure}

\subsection*{Organization}
Section~\ref{sec:petal} fixes conventions and describes petal permutations as arc presentations. Section~\ref{sec:moves} gives the precise definitions of all moves used, and explains how the Reidemeister moves are encoded by them in \cite{CGHS}. Section~\ref{sec:pages} compares the role of pages in arc presentations, closed braids and petal diagrams. Section~\ref{sec:computation} describes the computation, Section~\ref{sec:results} the results, and Section~\ref{sec:questions} lists open questions.

\section{Petal diagrams as arc presentations}\label{sec:petal}

Fix $N=2n+1$. A \emph{petal word} is a bijection $p\colon \Z/N\to\{0,1,\dots,N-1\}$, written $p=\word{p_0\,p_1\,\cdots\,p_{N-1}}$, considered up to cyclic shift of the indices. There are $(N-1)!$ petal words with $N$ petals. When we list words we use the lexicographically smallest cyclic rotation, which always begins with $0$, and index positions from $0$.

\begin{construction}\label{con:arc}
Use cylindrical coordinates $(r,\theta,z)$ on $\R^3$, and let the $z$-axis be the \emph{binding}. For $k\in\Z/N$ let $H_k$ be the open half-plane $\theta=2\pi k/N$, the $k$-th \emph{page}. Put
\[
c_j \;=\; j(n+1)\bmod N\qquad (j\in\Z/N).
\]
Let $v_j$ be the point of the binding at height $z=p_j$, and let $A_j$ be a simple arc in $H_k\cup\{\text{binding}\}$ with $k=c_j$, joining $v_j$ to $v_{j+1}$ and meeting the binding only in its endpoints. The \emph{knot of $p$} is $K(p)=\bigcup_j A_j$.
\end{construction}

Since each page contains exactly one arc, and an arc in a half-plane with given endpoints on its boundary is unique up to isotopy, $K(p)$ is well defined up to isotopy. It is an arc presentation in the sense of Cromwell \cite{Cromwell} with $N$ arcs. The pages are visited in the order $c_0,c_1,c_2,\dots$, that is, in steps of $n+1$.

\begin{lemma}\label{lem:petalarc}
$K(p)$ is isotopic to the knot represented by the petal diagram with petal permutation $p$, the heights being read in the order in which the knot traverses the strands through the multi-crossing.
\end{lemma}

\begin{proof}
Place the multi-crossing of a petal diagram at the origin of the projection plane, viewed along the $z$-axis, with petal $k$ centred at angle $2\pi k/N$ and bounded by two radial edges at angles $2\pi k/N\pm\pi/(2N)$. The strand at height $h$ is a straight segment through the origin. A line through the origin leaving petal $k$ along its counterclockwise edge arrives at angle $2\pi k/N+\pi/(2N)+\pi = 2\pi(k+n+1)/N-\pi/(2N)$. This is the clockwise edge of petal $k+n+1$. Hence, following the knot, consecutive petals differ by $n+1$, which matches the page sequence $c_j$.

The part of the knot between two consecutive passes through the multi-crossing, at heights $p_j$ and $p_{j+1}$, lies in a thin wedge around the angle $2\pi c_j/N$ that no other part of the knot enters, except along the $z$-axis at other heights. An isotopy supported in these wedges flattens each petal, together with the two half-strands bounding it, into the half-plane $H_{c_j}$. The result is $K(p)$. The converse construction is the same isotopy run backwards.
\end{proof}

\begin{remark}
This is the petal arc presentation of \cite[\S3]{Adams1}. The convention agrees with that of \cite{CGHS}. For instance, their permutations $\word{3\,1\,4\,2\,0}$ and $\word{6\,4\,2\,0\,3\,5\,1}$ give the trefoil and the figure-eight knot, and both of their moves preserve the knot type of $K(p)$ in all of the more than $1200$ random instances we tested. Conventions differing by the direction of the height axis, or of the traversal, replace $K(p)$ by its mirror image. Since the unknot is amphichiral, and since mirroring maps our move set $\M$ to itself (Remark~\ref{rem:mirror}), none of our results depends on these conventions.
\end{remark}

\begin{remark}\label{rem:grid}
Blowing up the binding turns $K(p)$ into a grid (rectangular) diagram \cite{Cromwell,Dynnikov,OSS}. Column $k$ is page $k$ and row $h$ is height $h$. The vertical segment in column $c_j$ joins rows $p_j$ and $p_{j+1}$, and the horizontal segment in row $p_{j+1}$ joins columns $c_j$ and $c_{j+1}$. Vertical segments cross over horizontal ones. Figure~\ref{fig:grid} shows the grid diagram of $\gamma$. This is the form we use in the computations.
\end{remark}

\begin{figure}[t]
\centering
\begin{tikzpicture}[scale=0.42,line width=0.8pt]
\draw[gray!35,line width=0.3pt] (-0.5,-0.5) grid (8.5,8.5);
\draw (0,2) -- (5,2);
\draw (5,4) -- (1,4);
\draw (1,7) -- (6,7);
\draw (6,3) -- (2,3);
\draw (2,8) -- (7,8);
\draw (7,6) -- (3,6);
\draw (3,1) -- (8,1);
\draw (8,5) -- (4,5);
\draw (4,0) -- (0,0);
\draw[white,line width=4pt] (0,0) -- (0,2);
\draw (0,0) -- (0,2);
\draw[white,line width=4pt] (5,2) -- (5,4);
\draw (5,2) -- (5,4);
\draw[white,line width=4pt] (1,4) -- (1,7);
\draw (1,4) -- (1,7);
\draw[white,line width=4pt] (6,7) -- (6,3);
\draw (6,7) -- (6,3);
\draw[white,line width=4pt] (2,3) -- (2,8);
\draw (2,3) -- (2,8);
\draw[white,line width=4pt] (7,8) -- (7,6);
\draw (7,8) -- (7,6);
\draw[white,line width=4pt] (3,6) -- (3,1);
\draw (3,6) -- (3,1);
\draw[white,line width=4pt] (8,1) -- (8,5);
\draw (8,1) -- (8,5);
\draw[white,line width=4pt] (4,5) -- (4,0);
\draw (4,5) -- (4,0);
\node[font=\tiny,left] at (-0.5,0) {0};
\node[font=\tiny,left] at (-0.5,1) {1};
\node[font=\tiny,left] at (-0.5,2) {2};
\node[font=\tiny,left] at (-0.5,3) {3};
\node[font=\tiny,left] at (-0.5,4) {4};
\node[font=\tiny,left] at (-0.5,5) {5};
\node[font=\tiny,left] at (-0.5,6) {6};
\node[font=\tiny,left] at (-0.5,7) {7};
\node[font=\tiny,left] at (-0.5,8) {8};
\node[font=\tiny,below] at (0,-0.5) {0};
\node[font=\tiny,below] at (1,-0.5) {1};
\node[font=\tiny,below] at (2,-0.5) {2};
\node[font=\tiny,below] at (3,-0.5) {3};
\node[font=\tiny,below] at (4,-0.5) {4};
\node[font=\tiny,below] at (5,-0.5) {5};
\node[font=\tiny,below] at (6,-0.5) {6};
\node[font=\tiny,below] at (7,-0.5) {7};
\node[font=\tiny,below] at (8,-0.5) {8};
\end{tikzpicture}
\caption{The grid diagram of the petal word $\gamma=\word{0\,2\,4\,7\,3\,8\,6\,1\,5}$ (columns $=$ pages, rows $=$ heights; vertical segments pass over horizontal ones). The knot visits the columns in the order $0,5,1,6,2,7,3,8,4$, which is forced by the petal structure. Only the row data, i.e.\ the word, is free.}
\label{fig:grid}
\end{figure}
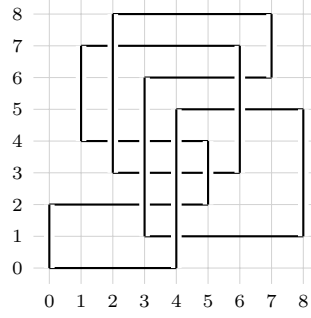

\section{The moves}\label{sec:moves}

We now define all moves precisely. Let $p=\word{p_0\,\cdots\,p_{2n}}$ be a petal word.

\subsection{Trivial petal addition and deletion \texorpdfstring{($\T^\pm$)}{(T)}}
For $m\in\Z$ let $g_m(a)=a$ if $a<m$ and $g_m(a)=a+2$ if $a\ge m$. A \emph{trivial petal addition} \cite[\S2.1]{CGHS} replaces $p$ by
\[
\begin{gathered}
\word{g_m(p_0)\,\cdots\,g_m(p_j)\;\;m\;\;m{+}1\;\;g_m(p_{j+1})\,\cdots\,g_m(p_{2n})}\\
\text{or}\quad
\word{g_m(p_0)\,\cdots\,g_m(p_j)\;\;m{+}1\;\;m\;\;g_m(p_{j+1})\,\cdots\,g_m(p_{2n})}
\end{gathered}
\]
for some position $j$ and some $0\le m\le 2n+1$. The number of petals increases by $2$. The inverse, a \emph{trivial petal deletion} $\Tm$, applies exactly when two cyclically adjacent letters of $p$ are consecutive integers $m,m+1$ (in either order). One deletes them and lowers all letters larger than $m+1$ by $2$.

In the arc presentation $K(p)$, two consecutive vertices at heights $m,m+1$ with no vertex between them are joined by a short arc. Shrinking it merges them into one vertex whose two arcs lie in pages $c_{j-1}$ and $c_{j+1}=c_{j-1}+2(n+1)=c_{j-1}+1$ of the presentation (here $N=2n+1$ is the number of petals before the deletion, and $2(n+1)\equiv1 \bmod N$). These are adjacent pages, so the two arcs can be rotated into one page. This is a direct way to see that $\T^\pm$ is an isotopy; a proof via stem diagrams is \cite[Cor.~5]{CGHS}.

\subsection{Crossing exchange \texorpdfstring{($\X$)}{(X)}}
Choose a \emph{basepoint position} $b\in\Z/N$. The remaining $2n$ positions are partitioned into $n$ cyclically consecutive pairs
\[
P_b=\bigl\{\{b+1,b+2\},\ \{b+3,b+4\},\ \dots,\ \{b+2n-1,b+2n\}\bigr\}.
\]
These are the left-pairs of \cite[\S2.2]{CGHS} for one choice of word representing $p$, and the right-pairs for another; letting $b$ vary covers both. For a pair $\Delta=\{i,i'\}\in P_b$ let $E(\Delta)=\{p_i,p_{i'}\}$ be its set of values. Suppose that there are integers $m$ and $w\ge m+2$ and two pairs $\Delta,\Delta'\in P_b$ with
\[
E(\Delta)=\{m,\,w+1\},\qquad E(\Delta')=\{m+1,\,w\},
\]
and such that for every other pair $\Lambda\in P_b$ the set $E(\Lambda)$ is either contained in or disjoint from the interval $[m+2,w-1]$. Then the \emph{crossing exchange} \cite[\S2.3]{CGHS} replaces $p$ by the word obtained by interchanging the letters $m\leftrightarrow m+1$ and $w\leftrightarrow w+1$. The order of the two letters inside $\Delta$ and $\Delta'$ is irrelevant. The basepoint letter $p_b$ never belongs to $\Delta$ or $\Delta'$. A crossing exchange preserves the number of petals and is an isotopy \cite[Lemma~6]{CGHS}.

In our tables we record a crossing exchange as $\X(b,m,w)$, where $b$ is the basepoint position in the word as displayed.

\begin{example}[{\cite[Fig.~4]{CGHS}}]
Let $p=\word{4\,2\,0\,6\,5\,1\,3}$ and $b=6$. Then $P_6=\{\{0,1\},\{2,3\},\{4,5\}\}$ with value sets $\{2,4\},\{0,6\},\{1,5\}$. With $m=1$, $w=4$, $\Delta=\{4,5\}$, $\Delta'=\{0,1\}$ the hypotheses hold, since $[3,3]$ is disjoint from $\{0,6\}$. Thus $\X(6,1,4)$ gives $\word{5\,1\,0\,6\,4\,2\,3}$.
\end{example}

In arc-presentation terms, interchanging two letters $h,h+1$ swaps two adjacent vertices on the binding. The nesting hypothesis ensures that the two swaps $m\leftrightarrow m+1$ and $w\leftrightarrow w+1$ undo each other's effect on the knot type. Neither swap alone is an isotopy in general; see Proposition~\ref{prop:noexchange}.

\subsection{Rotation of heights \texorpdfstring{($\Rot$)}{(R)}}
$\Rot(p)=\word{p_0+1\,\cdots\,p_{2n}+1}$, with letters taken modulo $N$.

\begin{lemma}
$K(\Rot(p))$ is isotopic to $K(p)$.
\end{lemma}
\begin{proof}
Compactify $\R^3$ to $S^3$. The binding becomes an unknotted circle and the pages become the pages of an open book with disk pages. An arc presentation in $S^3$ is determined up to isotopy by the cyclic order of its vertices on the binding circle and the cyclic order of its pages \cite{Cromwell}. Sliding the top vertex, at height $N-1$, up through $\infty$ and back from below, carrying its two arcs along in their pages, changes the linear order of the heights by the cyclic shift $h\mapsto h+1 \pmod N$. It leaves the pages unchanged.
\end{proof}

\subsection{Rotation by \texorpdfstring{$\pi$}{pi} \texorpdfstring{($\Sym$)}{(S)}}
$\Sym(p)=\word{N{-}1{-}p_{2n}\;\;N{-}1{-}p_{2n-1}\;\cdots\;N{-}1{-}p_{0}}$, i.e.\ reverse the word and complement every letter.

\begin{lemma}
$K(\Sym(p))$ is isotopic to $K(p)$.
\end{lemma}
\begin{proof}
Apply the rotation $\rho(x,y,z)=(x,-y,-z)$, by $\pi$ about the $x$-axis. It maps the height $h$ to $-h$, which after relabelling is $N-1-h$, and the page $H_k$ to $H_{-k}$. The arc $A_j$ in page $c_j=j(n+1)$ is mapped to an arc in page $-j(n+1)=c_{-j}$ joining heights $N-1-p_j$ and $N-1-p_{j+1}$. Hence $\rho(K(p))=K(q)$ with $q_{-j}=N-1-p_{j+1}$, which is $\Sym(p)$ up to cyclic shift.
\end{proof}

\begin{remark}\label{rem:mirror}
By the same computation, reversing the word (reflection $\theta\mapsto-\theta$) and complementing the letters (reflection $z\mapsto -z$) each give the mirror image. We verified this on the trefoil using the writhe of a minimal diagram. Mirroring conjugates each of $\T^\pm,\X,\Rot,\Sym$ to a move of the same type, so the set of $\M$-hard words and the excess are invariant under mirroring. One checks that $\Sym\Rot\Sym^{-1}=\Rot^{-1}$, so $\langle\Rot,\Sym\rangle$ is dihedral of order $2N$.
\end{remark}

By the theorem of \cite{CGHS}, $\Rot$ and $\Sym$ can be expressed through $\T^\pm$ and $\X$, but in general only through diagrams with more petals. This is why we add them explicitly.

\subsection{How the Reidemeister moves are encoded}\label{subsec:CGHS}
We recall the structure of the proof of the completeness theorem.

\begin{theorem}[{Colton--Glover--Hughes--Sandberg \cite[Thm.~1]{CGHS}}]\label{thm:CGHS}
Two petal permutations represent isotopic knots if and only if they are related by a finite sequence of trivial petal additions, trivial petal deletions and crossing exchanges.
\end{theorem}

The proof passes through \emph{stem diagrams}. A stem diagram is a classical oriented diagram $D$ together with a properly embedded line $\alpha$ (the \emph{axis}) splitting the plane into a left side $L$ and a right side $R$, and a basepoint $s\in D\cap\alpha$. These are required to satisfy the following condition: travelling from $s$, every crossing in $L$ is first met along its under-strand, and every crossing in $R$ along its over-strand. Viewing the corresponding embedding from the direction of the axis gives a petal diagram whose petals correspond to the arcs of $D$ cut along $\alpha$ \cite[Lemma~2]{CGHS}. In our language, $\alpha$ lifts to the binding of the arc presentation.

The moves are then interpreted as follows.
\begin{itemize}[leftmargin=2em]
\item $\T^+$ is a finger move that pushes an arc of $D$ across the axis along a path $\gamma$, performing Reidemeister~II moves where $\gamma$ meets $D$ \cite[Fig.~8]{CGHS}.
\item $\X$ inserts two crossings near the axis between two nested strands on the same side \cite[Fig.~9]{CGHS}. Followed by a Reidemeister~II move, it slides a crossing between the strands at levels $m,m+1$ to the strands at levels $w,w+1$ \cite[Fig.~11]{CGHS}.
\end{itemize}
The proof then has two steps. First, changes of the axis and of the basepoint are realized by $\T^\pm$ and $\X$ \cite[Lemmas~7, 8 and Prop.~9]{CGHS}. Second, for each move of Polyak's minimal generating set $\{\Omega1a,\Omega1b,\Omega2a,\Omega3a\}$ of oriented Reidemeister moves \cite{Polyak}, one chooses axes and basepoints before and after the move so that the petal permutation is unchanged \cite[Lemma~10]{CGHS}. For $\Omega1$, $\Omega2$ and one subcase of $\Omega3$ the axis is chosen away from the move. For the other subcase of $\Omega3$, different axes are used before and after.

\begin{remark}
Results A--C do not rely on Theorem~\ref{thm:CGHS}. They use only that each move in $\M$ preserves the knot type. For $\T^\pm$ and $\X$ this is the easier half of \cite{CGHS}, proved there in Corollary~5 and Lemma~6; for $\Rot,\Sym$ it was shown above. The completeness statement is used only to guarantee that untangling sequences exist at all, and our explicit sequences make even this unnecessary for the examples.
\end{remark}

\section{Pages in arc presentations, braids and petal diagrams}\label{sec:pages}

All the diagram types compared here can be viewed in the standard open book of $S^3$: the binding is an unknotted circle $\ell$, and the pages are disks bounded by $\ell$. They differ in how the knot sits relative to the pages, and this determines which moves are available.

\subsection{Classical diagrams}
Classical diagrams carry no open book structure. Reidemeister moves change the crossing number by $0$ or $\pm1$ (move I) or $\pm 2$ (move II), and a diagram of the unknot is \emph{hard} if every sequence of Reidemeister moves to the trivial diagram must at some point increase the crossing number. The classical example is due to Goeritz \cite{Goeritz}; further early examples of non-trivial projections of the trivial knot were given by Ochiai \cite{Ochiai}. Kauffman and Lambropoulou \cite{KL,KL2} constructed infinite families of hard unknots from rational tangles, identified the smallest ones, and introduced the more general notion of hard diagrams of arbitrary knots and links. Henrich and Kauffman \cite{HK} gave upper bounds on the number of crossings that may be needed in intermediate diagrams, using Dynnikov's theorem and Lackenby's bound on Reidemeister moves; Jab{\l}onowski \cite{Jab} determined minimal hard diagrams of classical unlinks. Burton et al.\ \cite{BCLMMSSS} exhibit unknot diagrams requiring at least three extra crossings in $S^2$, and observed that many ``standard'' hard unknots need only one. Using reinforcement learning, Applebaum et al.\ \cite{ABDEJLTZ} found millions of hard unknot diagrams, and found that several examples from the literature are not hard in the strict sense. Lunel, de Mesmay and Spreer \cite{LMS} constructed hard diagrams of split links. Cantarella and Schumacher \cite{CS} showed that many hard unknots become easy after a change of projection, a phenomenon related to the fact, exploited in this paper, that the diagram class one restricts to strongly affects the difficulty of untangling. On the quantitative side, the number of Reidemeister moves needed to untangle an unknot diagram is bounded above exponentially by Hass and Lagarias \cite{HL} and polynomially by Lackenby \cite{Lackenby}, and can be at least quadratic by Hass and Nowik \cite{HN}; for the computational complexity of related unknotting problems see \cite{MRST}. In any case no monotonicity holds.

\subsection{Arc presentations and rectangular diagrams}
In an arc presentation \cite{Cromwell} the knot meets $\ell$ in finitely many vertices and each page in at most one arc. Two pieces of data are independent: the cyclic order of the vertices on $\ell$ (rows of the rectangular diagram) and the cyclic order of the occupied pages (columns). The elementary moves \cite{Cromwell,Dynnikov} are the following.
\begin{enumerate}[leftmargin=2em]
\item[(i)] \emph{Exchange of adjacent pages}, allowed when the endpoint pairs of their two arcs do not interleave on $\ell$.
\item[(ii)] \emph{Exchange of adjacent vertices}, allowed when the page pairs of the two arcs at each vertex do not interleave in the cyclic order of pages.
\item[(iii)] \emph{Stabilization and destabilization}, changing the number of arcs by $\pm1$.
\end{enumerate}
In the rectangular picture, (i) and (ii) are the same move with the roles of the two coordinates interchanged.

\begin{theorem}[{Dynnikov \cite{Dynnikov}}]
Every rectangular diagram of the unknot can be transformed into the trivial diagram by a finite sequence of exchange moves and destabilizations. In particular there are no hard rectangular diagrams of the unknot.
\end{theorem}

The proof uses a spanning disk in general position with respect to the open book and the combinatorics of its foliation, in the tradition of Birman and Menasco \cite{BM}. It uses both kinds of exchange in an essential way.

\subsection{Closed braids}
A closed braid is transverse to all pages and meets each page in the same number $b$ of points. The pages play the role of time, and nothing is attached to the binding. Birman and Menasco \cite{BM} proved that a closed braid representing the unlink can be simplified to the trivial representative by braid isotopy, exchange moves and destabilizations, without increasing the braid index. This is another monotonic simplification theorem in the open book.

\subsection{Other restricted classes of diagrams}
Hardness also depends on the class of diagrams in which one is required to stay. Brugall\'e, Koseleff and Pecker \cite{BKP} studied \emph{trigonal} diagrams (Conway's open form, i.e.\ closures of $3$-strand braid-like diagrams) of two-bridge links, with \emph{slide isotopies} that keep all intermediate diagrams trigonal and never increase the number of crossings. They proved that for the two-bridge torus links and the generalized twist links every trigonal diagram can be simplified in this way to an alternating one. For all other two-bridge links they found \emph{awkward} trigonal diagrams for which this is impossible, in the spirit of the hard diagrams of \cite{KL2}. Petal diagrams provide another restricted class in which hardness occurs, already for the unknot.

\subsection{Petal diagrams}
By Lemma~\ref{lem:petalarc}, a petal diagram is an arc presentation with an odd number $N=2n+1$ of arcs. Its page order is not free: the $j$-th arc along the knot lies in page $c_j=j(n+1)$. Only the order of the vertices on the binding, i.e.\ the petal word, can vary. Thus, compared with general arc presentations, one of the two independent cyclic orders has been frozen and tied to the traversal of the knot.

\begin{proposition}\label{prop:noexchange}
Let $p$ be a petal word with $N=2n+1\ge5$ petals, viewed as the arc presentation $K(p)$.
\begin{enumerate}[leftmargin=2em]
\item[(a)] Let $0\le h\le N-2$. If the letters $h$ and $h+1$ are not cyclically adjacent in $p$, then the page pairs at the vertices of heights $h$ and $h+1$ interleave, so exchanging these vertices is not an exchange move of type (ii). If they are adjacent, then $\Tm$ applies to $p$.
\item[(b)] No page exchange (i), stabilization or destabilization transforms $K(p)$ into a petal arc presentation.
\end{enumerate}
\end{proposition}

\begin{proof}
(a) The vertex $v_j$ is the endpoint of the arcs $A_{j-1}$ and $A_j$, in pages $c_{j-1}=c_j+n$ and $c_j$. So its page pair is the chord $\chi_j=\{c_j,\,c_j+n\}$ of the $N$-gon $\Z/N$. Consider two such chords $\{a,a+n\}$ and $\{a+k,a+k+n\}$ with $1\le k\le 2n$.
\begin{itemize}[leftmargin=2em]
\item If $1\le k\le n-1$, then $a+k$ lies strictly inside the arc from $a$ to $a+n$, while $a+k+n$ lies outside it. So the chords interleave.
\item If $n+2\le k\le 2n$, then $a+k$ lies outside that arc, while $a+k+n=a+(k-n-1)$ lies inside. So the chords interleave.
\item If $k\in\{n,n+1\}$, the chords share an endpoint.
\end{itemize}
Since $c_i-c_j=(i-j)(n+1)$, sharing an endpoint happens exactly when $i-j=\pm1$, i.e.\ when the two vertices are consecutive along the knot. If $p_i=h$ and $p_j=h+1$ are consecutive along the knot, the letters $h,h+1$ are adjacent in $p$, and $\Tm$ applies.

(b) Stabilizations and destabilizations change the number of arcs by one, and petal arc presentations have an odd number of arcs. Now suppose a page exchange swaps two adjacent pages $k,k+1$. The pages visited along the knot become $\tau(c_j)$, where $\tau$ is the transposition $(k\ k{+}1)$ of $\Z/N$. For the result to be a petal arc presentation, for some starting point and orientation, $j\mapsto\tau(c_j)$ would have to be an affine function $j\mapsto u+ js$ with $s\in\{n,n+1\}$. Then $\tau$ itself would be affine. But an affine bijection of $\Z/N$ fixing the two points $k+2,k+3$ (distinct from $k,k+1$ because $N\ge5$) is the identity, since they differ by the unit $1$. This contradicts $\tau\neq\mathrm{id}$.
\end{proof}

Thus no elementary move of Cromwell and Dynnikov ever connects two different petal arc presentations of the same size, except the swap of two letters that are adjacent in the word, which factors through $\Tm$ and $\Tp$. In particular, a crossing exchange is a composite of two vertex swaps ($m\leftrightarrow m+1$ and $w\leftrightarrow w+1$), neither of which is an exchange move. The freedom used in Dynnikov's argument is therefore unavailable. Petal-number-preserving moves are scarce, and in the examples of Result~B there are none at all.

\begin{table}[t]
\centering
\small
\begin{tabular}{@{}p{2.2cm}p{2.1cm}p{3.1cm}p{2.1cm}p{3.1cm}@{}}
\toprule
 & classical & arc pres./rectangular & closed braid & petal \\
\midrule
binding & --- & meets $K$ in vertices & disjoint from $K$ & meets $K$ in $2n+1$ vertices\\
pages & --- & $\le 1$ arc each; order free & $b$ points each & $1$ arc each; order forced \\
free data & planar diagram & vertex order, page order & braid word & vertex order only \\
complexity & crossings & arcs & strands $b$ & petals $2n+1$ \\
monotone for unknot & no & yes \cite{Dynnikov} & yes \cite{BM} & no (Res.~B) \\
\bottomrule
\end{tabular}
\caption{The role of the open book for four kinds of diagrams.}
\label{tab:pages}
\end{table}

\section{The computation}\label{sec:computation}

\subsection{Reduction to connected components}
Fix $N$ and consider the graph $G_N$ whose vertices are the $(N-1)!$ petal words with $N$ petals and whose edges are the moves $\X,\Rot,\Sym$. Its connected components are closed under all petal-number-preserving moves of $\M$.

\begin{lemma}\label{lem:criterion}
Suppose that no unknot word with fewer than $N$ petals is $\M$-hard. Then an unknot word $p$ with $N$ petals is $\M$-hard if and only if no word in the component of $p$ in $G_N$ admits a move $\Tm$.
\end{lemma}

\begin{proof}
A sequence of moves starting at $p$ that never exceeds $N$ petals stays in the component of $p$ until it first applies $\Tm$. If some word in the component admits $\Tm$, we reach a word with $N-2$ petals, which can then be untangled by hypothesis without exceeding $N-2$ petals.
\end{proof}

For $N=3$ every word admits $\Tm$, since all pairs of letters are adjacent. By induction it therefore suffices, for each $N$, to compute the components of $G_N$ (by union--find) and the knot types of the components containing no word that admits $\Tm$.

\subsection{Certification of knot types}
For a word $p$ we form the grid diagram of Remark~\ref{rem:grid} and its planar diagram code. The knot type is constant on components, so it suffices to certify it for one representative of each component.
\begin{itemize}[leftmargin=2em]
\item \emph{Unknots} are certified by SnapPy's \cite{SnapPy} diagram simplification. This applies Reidemeister moves and terminates in a diagram with no crossings, which is a proof of unknottedness.
\item \emph{Knotted} components are certified by a non-trivial Alexander polynomial. We compute it independently of SnapPy, from the Wirtinger presentation of the simplified diagram by Fox calculus, as a minor of the Alexander matrix.
\end{itemize}
For $N\le 9$ every component falls into one of these two cases. As sanity checks, the code reproduces the trefoil and figure-eight examples of \cite{CGHS} and their crossing exchange of \cite[Fig.~4]{CGHS}. We also tested, on more than $1200$ random words with $7$, $9$ and $11$ petals, that $\T^\pm$, $\X$ and $\Rot$ preserve the knot type, as they must.

\subsection{Excess}
To bound the excess we search the graph of all words reachable from a given word by moves of $\M$ through words with at most $L$ petals, for $L=9,11,13$. Words with fewer petals are explored first. Whenever the search terminates without reaching a word with at most $7$ petals, the reachable set is finite and closed, and the excess exceeds $L-9$. When it reaches such a word, we complete the path by the monotone procedure of Result~A and record it.

The ancillary file \texttt{petal\_hard.py} reproduces every number in this paper in about two minutes on a laptop (Python~3, SnapPy, SymPy).

\section{Results}\label{sec:results}

\begin{table}[t]
\centering
\begin{tabular}{@{}rrrrrr@{}}
\toprule
$N$ & words & components of $G_N$ & unknot words & unknot components & $\M$-hard words\\
\midrule
3 & 2 & 2 & 2 & 2 & 0\\
5 & 24 & 7 & 22 & 5 & 0\\
7 & 720 & 50 & 562 & 28 & 0\\
9 & 40320 & 1443 & 24992 & 579 & 108\\
\bottomrule
\end{tabular}
\caption{Petal words with $N$ petals (up to cyclic rotation). All knot types are certified as described in Section~\ref{sec:computation}.}
\label{tab:counts}
\end{table}

Table~\ref{tab:counts} summarizes the enumeration. For $N=5$ and $N=7$ every component of the graph generated by $\X$ alone, without $\Rot$ and $\Sym$, that represents the unknot contains a word admitting $\Tm$. Together with Lemma~\ref{lem:criterion} applied to this smaller move set, this establishes Result~A. For $N=9$, Lemma~\ref{lem:criterion} and Table~\ref{tab:counts} give the $108$ hard words of Result~B. They are precisely the unknot components of $G_9$ consisting of a single $\langle\Rot,\Sym\rangle$-orbit of size $18$ with no applicable $\X$ or $\Tm$. Their orbit representatives are
\[
\begin{array}{lll}
\alpha=\word{0\,2\,4\,1\,5\,7\,3\,8\,6}, & \beta=\word{0\,2\,5\,3\,8\,1\,7\,4\,6}, & \gamma=\word{0\,2\,4\,7\,3\,8\,6\,1\,5},\\[2pt]
\bar\alpha=\word{0\,2\,6\,1\,8\,4\,7\,5\,3}, & \bar\beta=\word{0\,2\,6\,4\,7\,1\,8\,5\,3}, & \bar\gamma=\word{0\,2\,6\,1\,7\,5\,3\,8\,4},
\end{array}
\]
where the bar denotes the mirror image (word reversal). For instance, in $\gamma$ the cyclically adjacent differences are $2,2,3,4,5,2,5,4,5$, none of them $\pm1$ or $\pm 8$. One checks by hand from the definition that no basepoint position $b$ yields pairs with value sets $\{m,w+1\}$ and $\{m+1,w\}$ satisfying the nesting condition.

\begin{remark}\label{rem:TX}
For the smaller move set $\{\T^\pm,\X\}$ of \cite{CGHS}, the same computation finds $264$ hard unknot words with $9$ petals, in $222$ components of the graph generated by $\X$. Of these words, $180$ admit no crossing exchange at all. Result~A holds for this move set as well.
\end{remark}

\subsection{Untangling sequences}
The following sequences prove the upper bounds in Result~C. Each line shows a word, in its lexicographically minimal cyclic rotation, followed by the move applied to it. $\Tp(j;\,a\,b)$ inserts the letters $a\,b$ after position $j$ (after raising letters $\ge\min(a,b)$ by $2$). $\Tm(a\,b)$ deletes the adjacent letters $a,b$. $\X(b,m,w)$ is as in Section~\ref{sec:moves}. Every word in these sequences is certified to be an unknot. None of the sequences uses $\Rot$ or $\Sym$.

\medskip
\noindent\emph{The word $\alpha$ (maximum $11$ petals).}
\[
\begin{array}{ll}
\word{0\,2\,4\,1\,5\,7\,3\,8\,6} & \Tp(6;\,5\,6)\\
\word{0\,2\,4\,1\,7\,9\,3\,5\,6\,10\,8} & \X(10,6,9)\\
\word{0\,2\,4\,1\,6\,10\,3\,5\,7\,9\,8} & \Tm(9\,8)\\
\word{0\,2\,4\,1\,6\,8\,3\,5\,7} & \X(5,0,6)\\
\word{0\,7\,8\,3\,5\,6\,1\,2\,4} & \Tm(7\,8)\\
\word{0\,3\,5\,6\,1\,2\,4} & \Tm(5\,6)\\
\word{0\,3\,1\,2\,4} & \Tm(1\,2)\\
\word{0\,1\,2} & \Tm(0\,1)\\
\word{0} &
\end{array}
\]

\noindent\emph{The word $\beta$ (maximum $11$ petals).}
\[
\begin{array}{ll}
\word{0\,2\,5\,3\,8\,1\,7\,4\,6} & \Tp(1;\,6\,7)\\
\word{0\,2\,6\,7\,5\,3\,10\,1\,9\,4\,8} & \X(2,4,7)\\
\word{0\,2\,6\,8\,4\,3\,10\,1\,9\,5\,7} & \Tm(4\,3)\\
\word{0\,2\,4\,6\,8\,1\,7\,3\,5} & \X(8,3,6)\\
\word{0\,2\,3\,7\,8\,1\,6\,4\,5} & \Tm(2\,3)\\
\word{0\,5\,6\,1\,4\,2\,3} & \Tm(5\,6)\\
\word{0\,1\,4\,2\,3} & \Tm(0\,1)\\
\word{0\,1\,2} & \Tm(0\,1)\\
\word{0} &
\end{array}
\]

\noindent\emph{The word $\gamma$ (maximum $13$ petals).}
\[
\begin{array}{ll}
\word{0\,2\,4\,7\,3\,8\,6\,1\,5} & \Tp(2;\,6\,7)\\
\word{0\,2\,4\,6\,7\,9\,3\,10\,8\,1\,5} & \Tp(6;\,7\,8)\\
\word{0\,2\,4\,6\,9\,11\,3\,7\,8\,12\,10\,1\,5} & \X(10,8,11)\\
\word{0\,2\,4\,6\,8\,12\,3\,7\,9\,11\,10\,1\,5} & \Tm(11\,10)\\
\word{0\,2\,4\,6\,8\,10\,3\,7\,9\,1\,5} & \X(10,3,6)\\
\word{0\,2\,3\,7\,8\,10\,4\,6\,9\,1\,5} & \Tm(2\,3)\\
\word{0\,5\,6\,8\,2\,4\,7\,1\,3} & \Tm(5\,6)\\
\word{0\,6\,2\,4\,5\,1\,3} & \Tm(4\,5)\\
\word{0\,4\,2\,1\,3} & \Tm(2\,1)\\
\word{0\,2\,1} & \Tm(2\,1)\\
\word{0} &
\end{array}
\]

To illustrate how the moves are checked, consider the third line of the sequence for $\gamma$. With $b=10$ the pairs are formed by positions $\{11,12\},\{0,1\},\{2,3\},\{4,5\},\{6,7\},\{8,9\}$, with value sets $\{1,5\},\{0,2\},\{4,6\},\{9,11\},\{3,7\},\{8,12\}$. For $m=8$, $w=11$ we have $\Delta=\{8,12\}$ and $\Delta'=\{9,11\}$, and all other value sets are disjoint from $[10,10]$. Interchanging $8\leftrightarrow9$ and $11\leftrightarrow12$ gives the fourth line.

For the lower bound for $\gamma$ in Result~C, the search through words with at most $11$ petals terminates with a closed set of $3186$ words: the $18$ words of the orbit of $\gamma$ and $3168$ words with $11$ petals. Every trivial petal deletion from any of these $11$-petal words lands back in the orbit of $\gamma$. Hence $e_\M(\gamma)=4$, and the same holds for its orbit and, by Remark~\ref{rem:mirror}, for the orbit of $\bar\gamma$. For $\alpha$ and $\beta$ hardness gives $e_\M\ge2$, and the sequences give $e_\M\le2$. This completes the verification of Result~C.

\begin{corollary}\label{cor:robust}
Let $\M'$ be any set of moves each of which is a composite of moves in $\M$ passing only through words with at most $N+2$ petals, where $N$ is the number of petals at the start of the composite. Then $\gamma$ is $(\M\cup\M')$-hard.
\end{corollary}
\begin{proof}
A sequence of moves of $\M\cup\M'$ from $\gamma$ through words with at most $9$ petals expands to a sequence of moves of $\M$ through words with at most $11$ petals. By Result~C such a sequence never leaves the closed set described there.
\end{proof}

In particular, the composite ``add a trivial petal, perform any number of crossing exchanges, delete a trivial petal'', viewed as a single petal-number-preserving move, removes the hardness of $\alpha$ and $\beta$ but not of $\gamma$.

\section{Questions}\label{sec:questions}

\begin{question}
Is the excess of unknot petal diagrams bounded by a function of the number of petals, and how fast must such a function grow? Does the maximal excess over $N$-petal diagrams tend to infinity?
\end{question}

\begin{question}
Is there a natural finite collection of \emph{local} petal-number-preserving moves, beyond crossing exchanges, with respect to which petal diagrams of the unknot admit monotonic simplification, in the spirit of Dynnikov's theorem? By Corollary~\ref{cor:robust}, such moves cannot all be composites passing through only two additional petals.
\end{question}

\begin{question}
How common are hard diagrams? With $9$ petals, $108$ of the $24992$ unknot words (about $0.43\%$) are $\M$-hard. What is the proportion for larger $N$, for instance in the Petaluma model \cite{EHLN1,EHLN2}?
\end{question}

\begin{question}
For a non-trivial knot $K$, are there petal diagrams of $K$ that cannot be reduced to a diagram with $p(K)$ petals without first increasing the number of petals? Dynnikov's theorem has no analogue of this kind for arbitrary knots in the rectangular setting either, so the question is whether petal diagrams behave even worse.
\end{question}

\subsection*{Acknowledgements}
The computations used SnapPy \cite{SnapPy} and SymPy. Claude AI was used to help with writing Python code and a draft of the paper.

\end{document}